\documentclass[12pt,a4paper,leqno,verbatim]{amsart}
 \def\registered{
 {\ooalign{\hfil\raise .00ex\hbox{\scriptsize R}\hfil\crcr\mathhexbox20D}}}

\newcounter{minutes}
\divide\time by 60
\newcounter{hours}
\multiply\time by 60 \addtocounter{minutes}{-\time}

\usepackage{amssymb}
\usepackage{hyperref}
\usepackage[T1]{fontenc}
\usepackage{graphicx}
\date{}
\newfont{\cyrilic}{wncyr10 scaled 1000}
\title[On generalized trigonometric functions with two parameters]{On generalized trigonometric functions \\ with two parameters }
\author[B. A. Bhayo]{Barkat Ali Bhayo}
\address{Department of Mathematics, University of Turku,
FI-20014 Turku, Finland} \email{barbha@utu.fi}

\author[M. Vuorinen]{Matti Vuorinen}
\address{Department of Mathematics, University of Turku,
FI-20014 Turku, Finland} \email{vuorinen@utu.fi}
\newcommand{\comment}[1]{}

\swapnumbers
\theoremstyle{plain}

\newtheorem{theorem}[equation]{Theorem}
\newtheorem{lemma}[equation]{Lemma}

\newtheorem{remark}[equation]{Remark}

\newtheorem{conjecture}[equation]{Conjecture}

\numberwithin{equation}{section}

\begin{document}
\font\fFt=eusm10 
\font\fFa=eusm7  
\font\fFp=eusm5  
\def\K{\mathchoice
{\hbox{\,\fFt K}}
{\hbox{\,\fFt K}}
{\hbox{\,\fFa K}}
{\hbox{\,\fFp K}}}

\def\thefootnote{}
\footnotetext{ \texttt{\tiny File:~\jobname .tex,
          printed: \number\year-\number\month-\number\day,
          \thehours.\ifnum\theminutes<10{0}\fi\theminutes}
} \makeatletter\def\thefootnote{\@arabic\c@footnote}\makeatother

\maketitle

\begin{abstract}
The generalized $p$-trigonometric and ($p,q$)-trigonometric functions were introduced by P. Lindqvist and S. Takeuchi, respectively. We prove some inequalities and present a few conjectures for the ($p,q$)-functions. \end{abstract}

{\bf 2010 Mathematics Subject Classification}: 41A17, 33C99, 33B99

{\bf Keywords and phrases}: Eigenfunctions of $p$-Laplacian, $\sin_{p,q}$, generalized trigonometric function.
\vspace{.5cm}

\section{Introduction}
During the past decade, many authors have studied the generalized trigonometric functions introduced by P. Lindqvist in a highly cited paper \cite{lp}. These so called $p$-trigonometric functions $p>1$,
which agree for $p=2$ with the familiar functions, have also been extended in various directions. The recent literature on these
functions includes several dozens of papers, see the bibliographies of
\cite{bem, dm, le}. Most recently, S. Takeuchi \cite{t} has taken one step further and investigated the $(p,q)$-trigonometric functions
depending on two parameters instead of one, and which for $p=q$
reduce to the $p$-functions of Lindqvist. See also D. E. Edmunds, P. Gurka and J. Lang \cite{egl}.

P. Dr\'abek and R. Man\'asevich \cite{dm} considered the following ($p,q$)-eigenvalue problem with the Dirichl\'et boundary condition.
Let $\phi_p(x)=|x|^{p-2}x.$ For $T,\lambda>0$ and $p,q>1$

\begin{equation*} \label{mybvp}
\left\{\begin{array}{lll}\displaystyle(\phi_p(u'))'+
\lambda\,\phi_q(u)=0,\quad t\in(0,T),\\
                   \displaystyle u(0)=u(T)=0.\end{array}\right.
\end{equation*}
They found the complete solution to this problem. This solution is
also given in \cite[Thm 2.1]{t}. In particular, for $T=\pi_{p,q}$ the function $u(t)\equiv \sin_{p,q}(t)$ is
a solution to this problem with $\lambda=\frac{p}{q}(p-1)$ where
$$\pi_{p,q}=\int_0^1{(1-t^q)^{-1/p}}\,dt=
\frac{2}{q}B\left(1-\frac{1}{p},\frac{1}{q}\right).$$
If $p=2\,,$ this eigenvalue-boundary value problem reduces
to the familiar boundary value problem whose solution is the
usual $\sin$ function.
Next, we will give an alternative equivalent definition of the
function $\sin_{p,q}\,,$ which is carried out in two
steps: in the first step we  define the inverse function of $\sin_{p,q}$, denoted by ${\rm arcsin}_{p,q}\,,$ and in the second step the function itself.
For $x\in[0,1]$, set
$$F_{p,q}(x)=\int_0^x{(1-t^q)}^{-1/p}\,dt\,.$$
Then $F_{p,q}:[0,1]\to [0,\pi_{p,q}/2]$ is an increasing homeomorphism, denoted by ${\rm arcsin}_{p,q}$, and therefore its inverse
$$\sin_{p,q}\equiv F_{p,q}^{-1}\,,$$
is defined on the the interval $[0,\pi_{p,q}/2]$.
Below we discuss also other related functions such as ${\arccos}_{p,q},$ and ${\rm arsinh}_{p,q}$.

For the expression of the function ${\rm arcsin}_{p,q}$
in terms of well-known
special functions we introduce some notation.
The \emph{Gaussian hypergeometric function} is the
analytic continuation to the slit plane $\mathbf{C}\setminus[1,\infty)$ of the series
$$F(a,b;c;z)={}_2F_1(a,b;c;z)=\sum^\infty_{n=0}\frac{(a,n)(b,n)}
{(c,n)}\frac{z^n}{n!},\qquad |z|<1,$$
for given complex numbers $a,b$ and $c$ with $c\neq0,-1,-2,\ldots$ .
Here $(a,0)=1$ for $a\neq 0$, and $(a,n)$ is the \emph{shifted factorial function} or the \emph{Appell symbol}
$$(a,n)=a(a+1)(a+2)\cdots(a+n-1)$$
for $n=1,2,\dots \,.$
The hypergeometric function has numerous special functions as its
special or limiting cases, see \cite{AS}.

For ${\rm Re}\, x>0$, ${\rm Re}\, y>0$, we define the classical \emph{gamma function} $\Gamma(x)$, the $\emph{psi function}$ $\psi(x)$ and the \emph{beta function} $B(x,y)$ by
$$\Gamma(x)=\int^\infty_0 e^{-t}t^{x-1}\,dt,\,\,\psi(x)=\frac{\Gamma^{'}(x)}{\Gamma(x)},\,\,
B(x,y)=\frac{\Gamma(x)\Gamma(y)}{\Gamma(x+y)},$$
respectively.

For $x\in I=[0,1]$ the function ${\rm arcsin}_{p,q}$
considered above can
be expressed in terms of the hypergeometric function as follows
$$
{\rm arcsin}_{p,q}\,x=\int^x_0(1-t^q)^{-1/p}dt=
x\,F\left(\frac{1}{p},\frac{1}{q};1+\frac{1}{q};x^q\right).$$
We also define
${\rm arccos}_{p,q}\,x={\rm arcsin}_{p,q}((1-x^p)^{1/q})$
(see \cite[Prop. 3.1]{egl}),
and
$${\rm arsinh}_{p,q}\,x=\int^x_0(1+t^q)^{-1/p}dt=
x\,F\left(\frac{1}{p},\frac{1}{q};1+\frac{1}{q};-x^q\right).$$

Their inverse functions are
$$\sin_{p,q}:(0,\pi_{p,q}/2)\to (0,1),\quad \cos_{p,q}:(0,\pi_{p,q}/2)\to (0,1), $$
$$\sinh_{p,q}:(0,m_{p,q})\to (0,1),\quad
m_{p,q}=\frac{1}{2^{1/p}}
F\left(1,\frac{1}{p};1+\frac{1}{q};\frac{1}{2}\right).$$
The significance of these expressions for this paper lies in the fact that we can now apply the vast available information about the
hypergeometric functions to the functions ${\rm arcsin}_{p,q}$
and ${\rm sin}_{p,q}\,.$

When $p=q$ these $(p,q)$-functions coincide with the $p$-functions
studied in the extensive earlier literature such as in \cite{bem,dm,le,bv},
and for $p=q=2$ they coincide with familiar elementary functions.

The main result of this paper
is the following theorem which refines our earlier results in \cite{bv}.

\begin{theorem}\label{d} For $p,q>1$ and $x\in(0,1)$, we have
\begin{enumerate}
\item $ x\left (1+\displaystyle\frac{x^q}{p(1+q)}\right)<
{\rm arcsin}_{p,q}\,x< \min\left\{\displaystyle\frac{\pi_{p,q}}{2}x,\displaystyle(1-x^q)^{-1/(p(1+q))}x\right\}$,
\item $ \displaystyle \left(\frac{x^p}{1+x^q}\right)^{1/p}\,L(p,q,x)
< {\rm arsinh}_{p,q}\,x
< \displaystyle\left(\frac{x^p}{1+x^q}\right)^{1/p}\,U(p,q,x),$\\
\end{enumerate}
where
$$L(p,q,x)=\max\left\{\left(1-\frac{qx^q}{p(1+q)(1+x^q)}\right)^{-1},
\left({x^q+1}\right)^{1/p} \left(\frac{p q+p+q x^q}{p (q+1) }\right)^{-1/q}\right\}\,,$$
and
$U(p,q,x)=\displaystyle\left(1-\frac{x^q}{1+x^q}\right)^{-q/(p (q+1))}\,.$

\end{theorem}

\begin{theorem}\label{thm2} For $p,q>1$, we have
\bigskip

\begin{enumerate}
\item $\displaystyle\left(\frac{p}{p-1}\right)^{1/q}\,
\alpha\left(\frac{1}{100},q\right)<
\pi_{p,q}<\left(\frac{pq+p-q}{q(p-1)}\right)^{1-1/q}
\left(\frac{p}{p-1}\right)^{1/q}\,\alpha\left(\frac{1}{30},q\right)$,
 \bigskip

$$\alpha(c,q)=\frac{2\sqrt{\pi }}{ (e\,q)^{1/q}}
   \sqrt[6]{\frac{q (q+4)+8}{q^3}+c},$$
\bigskip

\item $2^{1-2/p}\sqrt{\frac{\pi}{p}(4+p)}<\pi_{p',p}<
2^{1-2/p}\sqrt{\frac{\pi}{p}(4+p)+
\left(2\sqrt{\pi}\frac{\Gamma(3/4)}{\Gamma(1/4)}\right)^2}\,,$
\bigskip

\item $2^{2/p}\sqrt{\pi } \sqrt{\displaystyle\frac{5}{4}-\frac{1}{p}}
    < \pi_{p,p'} <
   2^{2/p}\sqrt{\pi }\displaystyle\frac{(2-1/p)^{3/2-1/p}}{\sqrt{e}(3/2-1/p)^{1-1/p}}$,
\end{enumerate}
where $p'=p/(p-1)$.
\end{theorem}

The area enclosed by the so-called $p$-circle
$$|x|^p+|y|^p=1$$
is $\pi_{p,p'}$, see \cite{lpe}. In particular,  $\pi_{2,2}= \pi= 3.14\dots \,.$

\section{Some relations for $(p,q)$-functions}

In this section we shall prove some inequalities for
the functions defined in Section 1.

\begin{lemma}\label{pf} Fix $p,q>1$ and $x\in(0,1)$.

(1) The functions
$$({\rm arcsin}_{p,q}(x^k))^{1/k},\quad ({\rm arsinh}_{p,q}(x^k))^{1/k}$$
are decreasing and increasing, respectively in $k\in(0,\infty)$.

(2) The function
$$k\,{\rm arcsin}_{p,q}(x/k)$$
is decreasing on $k\in(1,\infty)$.

(3) In particular, for $k\geq 1$
$$\sqrt[k]{{\rm arcsin}_{p,q}(x^k)}\leq {\rm arcsin}_{p,q}(x)
\leq ({\rm arcsin}_{p,q}\sqrt[k]{x})^k\,,$$
$$({\rm arsinh}_{p,q}\sqrt[k]{x})^k\leq {\rm arsinh}_{p,q}(x)
\leq \sqrt[k]{{\rm arsinh}_{p,q}(x^k)}\,,$$
$${\rm arcsin}_{p,q}(x/k)\leq ({\rm arcsin}_{p,q}(x))/k\,.$$

\end{lemma}

\begin{proof} Let
\[
G(x) = \int_0^x g(t)\, dt,\quad E=G(x^k)\,,\qquad
f(k) = \big( E \big)^{1/k}\,.
\]
We get
\[
f' = -E^{1/k} \log E \frac1{k^2} + \frac1k E^{1/k -1} E' x^k \log x
= \frac{E^{1/k}}{k^2} \left( -\log \frac{E}{x^k}  - \left(x^k  \frac{E'}{E} - 1 \right) \log \frac1{x^{k}}\right).
\]
If $g\ge 1$, then
\[
\frac{E}{x^k} = \frac{1}{x^k}\int_0^{x^k} g(t)\, dt \ge 1.
\]
If $g$ is increasing, then
\[
E' - \frac{E}{x^k} = g(x^k) - \frac{1}{x^k}\int_0^{x^k} g(t)\, dt \ge 0,
\]
so that $x^k  \frac{E'}{E} - 1 \ge 0$. Thus $f'\le 0$ under these assumptions.

For the case of arcsin$_{p,q}$, let $g(t)=(1-t^q)^{-1/p}$, so the
conditions are clearly satisfied. Next, for arsinh$_{p,q}$, we set $g(t)=(1+t^q)^{-1/p}$ and note that
 $g(t)\le 1$ for all $t>0$ and that $g$ is decreasing and thus conclude that $f'\ge 0$, and the claims in (1) follow.
 For (2), let
$$h(k)=k\,{\rm arcsin}_{p,q}\left(\frac{x}{k}\right)=
x\,F\left(\frac{1}{p},\frac{1}{q};1+\frac{1}{q};\left(\frac{x}{k}\right)^q\right)\,.$$
We get
$$h'(k)=-\frac{q\,x}{k\,p(1+q)}\left(\frac{x}{k}\right)^q\,F\left(\frac{1}{p},\frac{1}{q};1+\frac{1}{q};\left(\frac{x}{k}\right)^q\right)\leq 0,$$
and this completes the proof.

The proof of (3) follows from parts (1) and (2).
\end{proof}

\bigskip

\begin{theorem}\label{pf1} For $p,q>1$ and $r,s\in(0,1)$, the following inequalities hold:
\\
\begin{enumerate}
\item ${\rm arcsin}_{p,q}(r\,s)\leq \sqrt{{\rm arcsin}_{p,q}(r^2)\,{\rm arcsin}_{p,q}(s^2)}\leq
{\rm arcsin}_{p,q}(r)\,{\rm arcsin}_{p,q}(s)\,,$\\

\item ${\rm arsinh}_{p,q}(r)\,{\rm arsinh}_{p,q}(s)\leq\sqrt{{\rm arsinh}_{p,q}(r^2)\,{\rm arsinh}_{p,q}(s^2)}\leq
{\rm arsinh}_{p,q}(r\,s)\,.$\\

\end{enumerate}
\end{theorem}
\begin{proof}
Let $h(x) = \log f(e^x)$ where $f(u) >0\,.$  Then $h$ is convex (in the $C^2$ case)
when $h''\ge 0$, i.e.\ iff
\[
\frac f y (f' + y f'') \ge (f')^2,
\]
where $y=e^x$ and the function is evaluated at $y$. If $f''\ge 0$, then
$$\frac f y \ge f'(0)\,,$$
so a sufficient condition for convexity is
$f'(0) (f' + y f'') \ge (f')^2$. If $f''\le 0$, the reverse holds, so a
sufficient condition for concavity is $f'(0) (f' + y f'') \le (f')^2$.
Suppose
\[
f(x) = \int_0^x g(t)\, dt.
\]
Then $f' = g$ and $f'' = g'$. Then one easily checks that
$h$ is convex in case $g(t)$ is $(1-t^p)^{-1/q}$,
and concave for $g(t)$ equal to $(1+t^p)^{-1/q}$.
Now the proof follows easily from Lemma \ref{pf}.
\end{proof}

\begin{lemma}\label{rv}\cite[Thm 1.7]{kmsv} Let $f:\mathbb{R}_+\to \mathbb{R}_+$ be a differentiable function
and for $c\neq 0$ define
$$g(x)=\frac{f(x^c)}{(f(x))^c}\,.$$
We have the following
\begin{enumerate}
\item if $h(x)=\log(f(e^x))$ is a convex function, then $g(x)$ is monotone increasing for $c,x\in(0,1)$
and monotone decreasing for $c>1,\,x\in(0,1)$ or $c<0,\,x\in(0,1)$,\\
\item if $h(x)$ is a concave function, then $g(x)$ is monotone increasing for $c>1,\,x\in(0,1)$ or $c<0,\,x\in(0,1)$
and monotone decreasing for $c,x\in(0,1)$.
\end{enumerate}
\end{lemma}
We get the following lemma by the proof of Theorem \ref{pf1} and applying Lemma \ref{rv}.

\begin{lemma} Let $I=(0,1)$. For $p,q>1$ the function
$$g_1(x)=\frac{{\rm arcsin}_{p,q}(x^k)}{({\rm arcsin}_{p,q}(x))^k}$$
is increasing (decreasing) in $x\in I$ for $k\in I \quad(k\in \mathbb{R}\setminus [0,1])$, and
$$g_2(x)=\frac{{\rm arsinh}_{p,q}(x^k)}{({\rm arsinh}_{p,q}(x))^k}$$
is increasing (decreasing) in $x\in I$ for $k\in \mathbb{R}\setminus I \quad(k\in [0,1]).$ In particular, for $k\in I$,
$$\left(\frac{\pi_{p,q}}{2}\right)^{1-1/k}  \sqrt[k]{{\rm arcsin}_{p,q}(x^k)} \leq {\rm arcsin}_{p,q}(x)\,$$
$$\left(m_{p,q}\right)^{1-1/k}  \sqrt[k]{{\rm arsinh}_{p,q}(x^k)} \geq {\rm arsinh}_{p,q}(x).$$
The both inequalities reverse for $k\in \mathbb{R}\setminus [0,1]$.
\end{lemma}

\begin{lemma}\label{neu}\cite[Thm 2.1]{ne2} Let $f:\mathbb{R}_+\to \mathbb{R}_+$
be a differentiable, log-convex function and let $a\geq 1$. Then $g(x)=(f(x))^a/f(a\,x)$
 decreases on its domain. In particular, if $0\leq x\leq y\,,$ then the following inequalities
 $$\frac{(f(y))^a}{f(a\,y)}\leq\frac{(f(x))^a}{f(a\,x)}\leq (f(0))^{a-1}$$
 hold true. If $0<a\leq 1$, then the function $g$ is an increasing function on $\mathbb{R}_+$
and inequalities are reversed.
\end{lemma}

\begin{lemma} For $k,p,q>1$ and $r,s\in(0,1)$ with $r \geq s$, we have
\begin{eqnarray*}
\left(\frac{{\rm arcsin}_{p,q}(s)}{{\rm arcsin}_{p,q}(r)}\right)^{k}&\leq&
\frac{{\rm arcsin}_{p,q}(s^k)}{{\rm arcsin}_{p,q}(r^k)}\,,\\
\frac{{\rm arsinh}_{p,q}(s^k)}{{\rm arsinh}_{p,q}(r^k)}&\leq&
\left(\frac{{\rm arsinh}_{p,q}(s)}{{\rm arsinh}_{p,q}(r)}\right)^{k}
\,.
\end{eqnarray*}
\end{lemma}

\begin{proof} For $x>0$, the following functions
$$u(x)={\rm arcsin}_{p,q}(e^{-x})\,,\quad v(x)=1/{\rm arsinh}_{p,q}(e^{-x})$$  
are log-convex by the proof of Theorem \ref{pf1}. With the change of variables $e^{-x}=r$ the
inequalities follow from Lemma \ref{neu}.
\end{proof}

\begin{lemma}\label{ku}\cite[Thm 2, p.151]{ku}
Let $J\subset\mathbb{R}$ be an open interval, and let $f:J\to \mathbb{R}$
be strictly monotonic function. Let $f^{-1}:f(J)\to J$ be the inverse to $f$ then
\begin{enumerate}
\item if $f$ is convex and increasing, then $f^{-1}$ is concave,
\item if $f$ is convex and decreasing, then $f^{-1}$ is convex,
\item if $f$ is concave and increasing, then $f^{-1}$ is convex,
\item if $f$ is concave and decreasing, then $f^{-1}$ is concave.
\end{enumerate}
\end{lemma}

\begin{lemma} For $k,p,q>1$ and $r\geq s$, we have
\begin{eqnarray*}
\left(\frac{\sin_{p,q}(r)}{\sin_{p,q}(s)}\right)^k&\leq& \frac{\sin_{p,q}(r^k)}{\sin_{p,q}(s^k)}
,\quad\, r,s\in(0,1),\\
\left(\frac{\sinh_{p,q}(r)}{\sinh_{p,q}(s)}\right)^k&\geq& \frac{\sinh_{p,q}(r^k)}{\sinh_{p,q}(s^k)}
,\quad\, r,s\in(0,1)\,,
\end{eqnarray*}
inequalities reverse for $k\in(0,1)$.
\end{lemma}

\begin{proof} It is clear from the proof of Theorem \ref{pf1} that the functions
$$f(x)=\log({\rm arcsin_{p,q}}(e^{-x}))
,\, h(x)=\log(1/{\rm arsinh_{p,q}}(e^{x}))$$
are convex and decreasing. Then Lemma \ref{ku}(2) implies that
$$f^{-1}(y)=\log(1/\sin_{p,q}(e^y))
,\,h^{-1}(y)=\log(\sinh_{p,q}(e^{-y})),\,$$
are convex, now the result follows from Lemma \ref{neu}.
\end{proof}

Let $f:I\to (0,\infty)$ be continuous, where $I$ is a subinterval of $(0,\infty)$. Let $M$
and $N$ be any two mean values. We say that  $f$ is $MN$-convex (concave) if
$$f (M(x, y)) \leq (\geq)
N(f (x), f (y)) \,\, \text{ for \,\, all} \,\, x,y \in I\,.$$
For some properties of these
 functions, see \cite{avv2}. If $A(x,y)=(x+y)/2$ is the arithmetic mean, then we see that convex functions are $AA$-convex.

\begin{lemma}\label{anvu}\cite[Thm 2.4(1)]{avv2} Let $I=(0,b),\,0<b<\infty$, and let $f:I\to (0,\infty)$ be continuous. Then
$f$ is $AA$-convex (concave) if and only if $f$ is convex (concave),
where $A$ is the arithmetic mean.
\end{lemma}

\begin{lemma} For $p,q>1$, and $r,s\in(0,1)$, we have
\begin{enumerate}
\item $\displaystyle{\rm arcsin}_{p,q}\,r+{\rm arcsin}_{p,q}\,s \le 2\,{\rm arcsin}_{p,q}\,\left(\frac{r+s}{2}\right)\,,$\\
\item $\displaystyle{\rm sin}_{p,q}\,r+{\rm sin}_{p,q}\,s \ge 2\,{\rm sin}_{p,q}\,
\left(\frac{r+s}{2}\right)  \,,$\\
\item $\displaystyle{\rm arsinh}_{p,q}\,r+{\rm arsinh}_{p,q}\,s
\ge 2\,{\rm arsinh}_{p,q}\,\left(\frac{r+s}{2}\right) \,,$\\
\item $\displaystyle{\rm sinh}_{p,q}\,r+{\rm sinh}_{p,q}\,s \le 2\,{\rm sinh}_{p,q}\,
\left(\frac{r+s}{2}\right) \,.$ \\
\end{enumerate}
\end{lemma}

\begin{proof} Let $f(x)={\rm arcsin}_{p,q}\,x$ and $g(x)={\rm arsinh}_{p,q}\,x \,.$ Then
$$f'(x)=(1 - x^p)^{-1/p},\quad g'(x)=(1 + x^p)^{-1/p}$$
are increasing and decreasing, respectively. This implies that $f$ and $g$ are convex and concave. Now it follows from
Lemma \ref{ku}(1),(3) that $f^{-1}$ and $g^{-1}$ are concave and convex, respectively. The proof follows
from Lemma \ref{anvu}.
\end{proof}

\bigskip

For the following inequalities see \cite[Corollary 1.26]{bari} and
\cite[Corollary 1.10]{avv2}:
for all $x,y\in(0,\infty)$,
$$\cosh(\sqrt{x\,y})\leq \sqrt{\cosh(x)\,\cosh(y)}\,,$$
$$\sinh(\sqrt{x\,y})\leq \sqrt{\sinh(x)\,\sinh(y)}\,,$$
with equality if and only if $x=y$.

On the basis of our computer experiments we have arrived at the following conjecture.

\bigskip

\begin{conjecture} For $p,q\in(1,\infty)$ and $r,s\in(0,1)$, we have
\bigskip

\begin{enumerate}
\item ${\rm sin}_{p,q}(\sqrt{r\,s})\leq \sqrt{{\rm sin}_{p,q}(r){\rm sin}_{p,q}(s)} \,,$
\bigskip

\item ${\rm sinh}_{p,q}(\sqrt{r\,s})\geq \sqrt{{\rm sinh}_{p,q}(r){\rm sinh}_{p,q}(s)} \,.$
\bigskip
\end{enumerate}
\end{conjecture}

\begin{remark} {\rm
Edmunds, Gurka and Lang \cite[Prop. 3.4]{egl} proved that for
$x \in [0, \pi_{4/3,4}/4)  $
\begin{equation} \label{EGLidty}
    \sin_{4/3,4}(2x) = \frac{ 2 u v ^{1/3}}{(1+ 4 u^4  v^{4/3})^{1/2}} \,,\quad   u= \sin_{4/3,4}(x),\, v=  \cos_{4/3,4}(x)  \,.
\end{equation}
Note that in this case $q=p/(p-1)\,.$
The Edmunds-Gurka-Lang identity \eqref{EGLidty} suggests that in the particular case $q=p/(p-1)$
some exceptional behavior might be expected for $\sin_{p,q}\,.$
This special case might be worth of further investigation.}
\end{remark}

It seems to be a natural question to ask whether the addition formulas
for the trigonometric functions have counterparts for the $(p,q)$-functions. Our next results gives a subadditive inequality.

\begin{lemma} For $p,q>1$, the following inequalities hold
\begin{enumerate}
\item $\sin_{p,q}(r+s)\leq \sin_{p,q}(r)+\sin_{p,q}(s)\,,\quad\, r, s\in(0,\pi_{p,q}/4)\,,$
\item $\sinh_{p,q}(r+s)\geq \sinh_{p,q}(r)+\sinh_{p,q}(s)\,,\,\quad r, s\in(0,\infty)\,$.
\end{enumerate}
\end{lemma}

\begin{proof} Let $f(x)={\rm arcsin}_{p,q}(x)$,\,
$x\in(0,1)$. We get
$$f^{'}(x)=(1-x^q)^{-1/p}\,,$$
which is increasing, hence $f$ is convex.
Clearly, $f$ is increasing. Therefore
$$f_1=f^{-1}(y)=\sin_{p,q}(y)$$
is concave by Lemma \ref{ku}(1). This implies that
$f_1^{'}$ is decreasing. Clearly $f_1(0)=0$, and by
\cite[Thm 1.25]{avvb}, $f_1(y)/y$
is decreasing. Now it follows from \cite[Lem 1.24]{avvb}
that
$$f_1(r+s)\leq f_1(r)+f_1(s),$$
and (1) follows. The proofs of part (2) follows similarly.
\end{proof}

For $p,q>1,\,x\in(0,1)$ and $z\in(0,\pi_{p,q}/2)$, it follows from Theorem \ref{d} that
$${\rm arsinh}_{p,q}\,x < {\rm arcsin}_{p,q}\,x,\quad \sin_{p,q}\,z<\sinh_{p,q}\,z\,.$$

\begin{lemma} For $p,q>1,\,s\in(0,r]$ and $r\in(0,1)$, we have
\begin{enumerate}
\item $\displaystyle\frac{{\rm arcsin}_{p,q}\,s}{s} \le \displaystyle\frac{{\rm arcsin}_{p,q}\,r}{r},$\\
\item $\displaystyle\frac{{\rm arsinh}_{p,q}\,s}{\sqrt[p]{s^p/(1+s^q)}}  \le \displaystyle \frac{{\rm arsinh}_{p,q}\,r}{\sqrt[p]{r^p/(1+r^q)}},$\\
\item $\displaystyle\frac{{\rm arsinh}_{p,q}\,s}{s} \ge \displaystyle\frac{{\rm arsinh}_{p,q}\,r}{r}$.
\end{enumerate}
\end{lemma}

\begin{proof}
By definition we get
$$\frac{{\rm arcsin}_{p,q}\,s}{{\rm arcsin}_{p,q}\,r}=
\frac{s}{r}\frac{F(1/p,1/q;1+1/q;s^q)}{F(1/p,1/q;1+1/q;r^q)}\le \frac{s}{r}.$$
Similarly,
$$\frac{{\rm arsinh}_{p,q}\,s}{{\rm arsinh}_{p,q}\,r}=
\frac{s/(1+s^q)^{1/p}}{r/(1+r^q)^{1/p}}\frac{F(1,1/p;1+1/q;s^q/(1+s^q))}{F(1,1/p;1+1/q;r^q/(1+r^q))}
\le\left(\frac{s/(1+s^q)}{r/(1+r^q)}\right)^{1/p}$$
because $F(a,b,;c;x)$ is increasing in $x$. Part (3) follows from \cite[Theorem 1.25]{avvb}.
\end{proof}
\section{Proof of the main results}

For the following lemma see \cite[Theorems 1.19(10), 1.52(1), Lemmas, 1.33, 1.35]{avvb}.

\begin{lemma}\label{avb} \begin{enumerate}
\item For $a,b,c>0$, $c<a+b$, and $|x|<1$,
$$F(a,b;c;x)=(1-x)^{c-a-b}F(c-a,c-b;c;x)\,.$$
\item For $a,x\in(0,1)$, and  $b,c\in(0,\infty)$
$$F(-a,b;c;x)<1-\frac{a\,b}{c}\,x\,.$$
\item For $a,x\in(0,1)$, and  $b,c\in(0,\infty)$
$$F(a,b;c;x)+F(-a,b;c;x)>2\,.$$
\item Let $a,b,c\in(0,\infty)$ and $c>a+b$. Then for $x\in[0,1]$,
$$F(a,b;c;x)\leq \frac{\Gamma(c)\Gamma(c-a-b)}{\Gamma(c-a)\Gamma(c-b)}\,.$$
\item For $a,b>0$, the following function
$$f(x)=\frac{F(a,b;a+b;x)-1}{\log(1/(1-x))}$$
is strictly increasing from $(0,1)$ onto $(a\,b/(a+b),1/B(a,b))$.
\end{enumerate}
\end{lemma}

We will refer in our proofs to the following identity  \cite[15.3.5]{AS}:
\begin{equation}\label{as}
F(a,b;c;z)=(1-z)^{-b}F(b,c-a;c;-z/(1-z))\, .
\end{equation}

\begin{lemma}\label{cinq}\cite[Thm 2]{car} For $0<a<c,\,-\infty<x<1$ and $0<b<c$, the following inequality holds
$$\max\left\{\left(1-\frac{b\,x}{c}\right)^{-a},(1-x)^{c-a-b}\left(1-x+\frac{b\,x}{c}\right)^{a-c}\right\}
<F(a,b;c;x)<(1-x)^{-ab/c}\,.$$
\end{lemma}

{\bf Proof of Theorem \ref{d}.} For (1), we get from Lemma \ref{avb} (3),(2)
\begin{eqnarray*}
{\rm arcsin}_{p,q}\,x&=& x\,F\left(\frac{1}{p},\frac{1}{q};1+\frac{1}{q};x^q\right)\\
&>&\left(2-F\left(-\frac{1}{p},\frac{1}{q};1+\frac{1}{q};x^q\right)\right)x\\
&>&x\left (1+\displaystyle\frac{x^q}{p(1+q)}\right).
\end{eqnarray*}
The  second inequality of (1) follows easily from Lemmas \ref{cinq} and \ref{avb}(4).

For (2), if we replace $b=1/q,c-a=1/q, c=1+1/q$ and $x^q=z/(1-z)$ in (\ref{as}) then we get

\begin{eqnarray*}
{\rm arsinh}_{p,q}\,x&=& x\,F\left(\frac{1}{p},\frac{1}{q};1+\frac{1}{q};-x^q\right)\\
&=&\left(\frac{x^p}{1+x^q}\right)^{1/p}
F\left(1,\frac{1}{p};1+\frac{1}{q};\frac{x^q}{1+x^q}\right),
\end{eqnarray*}
now the proof follows easily form Lemma \ref{cinq}. \hfill $\square$

\begin{figure}
\includegraphics[width=12cm]{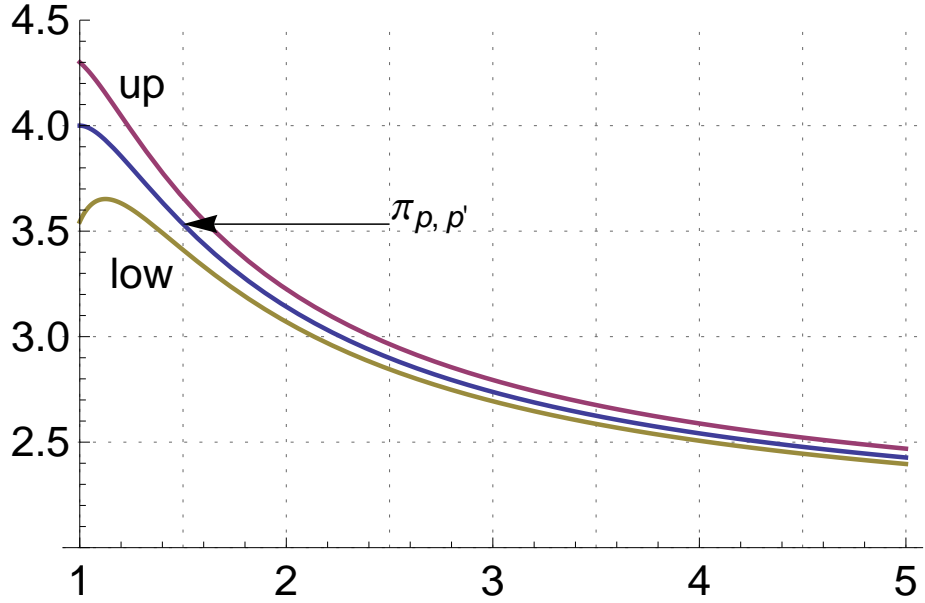}
\caption{We denote the lower and upper bounds of $\pi_{p,p'}$ by low and up.}
\end{figure}

\bigskip
For the following Lemma see \cite{al,ka,ke},\cite[Theorem 1]{kv},\cite{w}, respectively.

\begin{lemma}\label{ineq} The following relations hold,
\bigskip

\begin{enumerate}

\item $\displaystyle\sqrt{\pi}\left(\frac{x}{e}\right)^x\left(8x^3+4x^2+x+\frac{1}{100}\right)^{1/6}<\Gamma(1+x)$
$$<\sqrt{\pi}\left(\frac{x}{e}\right)^x\left(8x^3+4x^2+x+\frac{1}{30}\right)^{1/6},\quad x\geq 0,$$

\item $\displaystyle\left(x+\frac{s}{2}\right)^{1-s}<
    \displaystyle\frac{\Gamma(x+1)}{\Gamma(x+s)}<
\left(x-\frac{1}{2}+\left(\frac{1}{4}+s\right)^{1/2}\right)^{1-s}
,\quad x>0,\,s\in(0,1),$
\bigskip

\item $\displaystyle\frac{\Gamma(b)}{\Gamma(a)}<\displaystyle\frac{b^{b-1/2}}{a^{a-1/2}}e^{a-b},\quad b>a>0.$
\bigskip

\item $\displaystyle\left(\frac{x}{x+s}\right)^{1-s}\leq
\displaystyle\frac{\Gamma(x+s)}{x^s\Gamma(x)}\leq 1,
\quad x>0,\,s\in(0,1)$,
\bigskip
\end{enumerate}
\end{lemma}

{\bf Proof of Theorem \ref{thm2}}. If we let $x=1-1/p$ and $s=1/q$, then by definition
$$\pi_{p,q}=\frac{2\Gamma(x)\Gamma(1+s)}{\Gamma(s+x)}.$$
By Lemma \ref{ineq}(4) we get
$$\frac{2}{q}\,\Gamma(s)\left(\frac{p}{p-1}\right)^{1/q}<\pi_{p,q}
<\frac{2}{q}\,\Gamma(s)\left(\frac{pq+p-q}{q(p-1)}\right)^{1-1/q}\left(\frac{p}{p-1}\right)^{1/q}.$$
Now (1) follows if we use $\Gamma(1+x)=x\,\Gamma(x)$ and Lemma \ref{ineq}(1).
From \cite[6.1.18]{AS} we get
\begin{eqnarray*}
\pi_{p',p}&=&2\frac{\Gamma(1/p)\Gamma(1/p)}{p\Gamma(2/p)}=
2\frac{\Gamma(1/p)\Gamma(1+1/p)}{\Gamma(2/p)}\\
&=&2^{2-2/p}\sqrt{\pi}\frac{\Gamma(1+1/p)}{\Gamma(1/2+1/p)}\,,
\end{eqnarray*}
and (2) follows from Lemma \ref{ineq}(2) if we take $x=1/p$ and $s=1/2$. 

For (3), we see that
$$\pi_{p,p'}= \frac{2x\Gamma(x)^2}{\Gamma(2x)}
= \frac{2^{2-2x}\sqrt{\pi}x\Gamma(x)^2}{\Gamma(x)\Gamma(1/2+x)}\\
=\frac{2^{2-2x}\sqrt{\pi}\Gamma(1+x)}{\Gamma(1/2+x)},$$
and the lower bound follows from Lemma \ref{ineq}(2), and the upper bound follows if we replace $b=x+1$ and $a=x+s$ with $s=1/2$ in \ref{ineq}(3).
 \hfill $\square$

\begin{remark} {\rm For the benefit of an interested
reader we give an algorithm for the numerical
 computation of $\sin_{p,q}$  with the help of Mathematica$^{\tiny\textregistered}$  \cite{ru}.
The same method also applies to  $\sinh_{p,q}\,.$

\bigskip
\small
{\tt

arcsinp[p\_, q\_, x\_] := x * Hypergeometric2F1[1/p, 1/q, 1 + 1/q, x\^{ }p]

sinp[p\_, q\_, y\_] := x /. FindRoot[arcsinp[p, q, x] == y, \{x, 0.5 \}].

}
\normalsize

\bigskip

In the following tables we use the values of $p=2.5$ and $q=3$.}

\end{remark}
\bigskip

\begin{displaymath}
\begin{array}{|c|c|c|c|}
\hline
x&{\rm arcsin_{p,q}}(x)&{\rm arccos_{p,q}}(x)&{\rm arsinh_{p,q}}(x)\\
\hline

0.0000 & 0.0000 & 1.2748& 0.0000   \\

0.2500&0.2504&1.2048&0.2496\\

0.5000&0.5066&1.0688&0.4940\\

0.7500&0.7887&0.8536&0.7227\\

1.0000&1.2748&0.0000&0.9262\\

\hline
\end{array}
\end{displaymath}

\bigskip

\begin{displaymath}
\begin{array}{|c|c|c|c|}
\hline
x&{\rm sin_{p,q}}(x)&{\rm cos_{p,q}}(x)&{\rm sinh_{p,q}}(x)\\
\hline

0.0000 & 0.0000 & 1.0000 & 0.0000   \\

0.2500&0.2496&0.9937&0.2504\\

0.5000&0.4937&0.9500&0.5063\\

0.7500&0.7183&0.8309&0.7817\\

1.0000&0.8995&0.5943&0.1003\\

\hline
\end{array}
\end{displaymath}

\bigskip

{\sc Acknowledgements. } The work of the first author was supported
by the Academy of Finland, Project 2600066611 coordinated
by the second author. The authors are indebted to the referee for a
number of useful remarks.

\end{document}